\documentclass[11pt]{amsart}

\usepackage{graphicx}
\usepackage{hyperref}
\usepackage[noabbrev]{cleveref}
\usepackage{amsmath}
\usepackage{amsfonts}
\usepackage{amssymb}
\usepackage{array}
\usepackage{enumitem}
\usepackage{url}
\usepackage{tikz}
\usetikzlibrary{matrix,arrows}
\usepackage{tikz-cd}

\newtheorem{theorem}{Theorem}[section]

\newtheorem{corollary}[theorem]{Corollary}

\newtheorem{definition}[theorem]{Definition}

\newtheorem{example}[theorem]{Example}

\newtheorem{lemma}[theorem]{Lemma}

\newtheorem{proposition}[theorem]{Proposition}
\newtheorem{remark}[theorem]{Remark}

\newcommand{\BB}{\mathcal{B}}
\newcommand{\DD}{\mathcal{D}}
\newcommand{\EE}{\mathcal{E}}
\newcommand{\OO}{\mathcal{O}}

\newcommand{\FF}{\mathcal F}
\newcommand{\GG}{\mathcal G}

\newcommand{\PP}{\mathcal{P}}

\newcommand{\sh}{\mathbf{Sh}}

\newcommand{\restricted}[2]{#1{\mid}_{#2}}

\def\dom{\mathrm{dom}}

\def\Spec{\mathrm{Spec}}

\def\pt{\mathrm{pt}}
\def\patch{p}

\newcommand{\op}{\mathrm{op}}
\newcommand{\coh}{\mathrm{coh}}

\let\meet\wedge
\let\join\vee

\makeatletter
\newcommand{\subalign}[1]{%
  \vcenter{%
    \Let@ \restore@math@cr \default@tag
    \baselineskip\fontdimen10 \scriptfont\tw@
    \advance\baselineskip\fontdimen12 \scriptfont\tw@
    \lineskip\thr@@\fontdimen8 \scriptfont\thr@@
    \lineskiplimit\lineskip
    \ialign{\hfil$\m@th\scriptstyle##$&$\m@th\scriptstyle{}##$\crcr
      #1\crcr
    }%
  }
}
\makeatother

\let\phi\varphi
\let\epsilon\varepsilon

\renewcommand{\ast}{{\fontfamily{lmr}\selectfont\star}}

\title{Duality for noncommutative frames}
\author{Karin Cvetko-Vah} 
\address{University of Ljubljana, Faculty of Mathematics and Physics\\\mbox{Jadranska 19}\\1000 Ljubljana (Slovenia) \\ {\tt karin.cvetko@fmf.uni-lj.si}}
\author{Jens Hemelaer}
\thanks{Jens Hemelaer was supported in part by a PhD fellowship of the Research Foundation (Flanders) and in part by the University of Antwerp (BOF)}
\address{Department of Mathematics, University of Antwerp \\ 
 Middelheimlaan 1, \mbox{B-2020} Antwerp (Belgium) \\ {\tt jens.hemelaer@uantwerpen.be}}
\author{Lieven Le Bruyn}
\address{Department of Mathematics, University of Antwerp \\ 
 Middelheimlaan 1, \mbox{B-2020} Antwerp (Belgium) \\ {\tt lieven.lebruyn@uantwerpen.be}}

\AtBeginDocument{%
   \def\MR#1{}
}

\begin{document}

\begin{abstract}
We characterize the left-handed noncommutative frames that arise from sheaves on topological spaces. Further, we show that a general left-handed noncommutative frame $A$ arises from a sheaf on the dissolution locale associated to the commutative shadow of $A$. Both constructions are made precise in terms of dual equivalences of categories, similar to the duality result for strongly distributive skew lattices in \cite{kcv-priestley}.
\end{abstract}

\maketitle

\tableofcontents

\section{Introduction}
\label{sec:introduction}

Let $Y$ be a topological space. Following the terminology of Simmons in \cite{simmons}, we define the front topology on $Y$ to be the topology generated by all open and closed subsets of $Y$. Further, we write $Y_f$ for $Y$ equipped with the front topology. Now let $\EE$ be a sheaf on $Y_f$ with $\EE(Y_f) \neq \varnothing$ and consider the set 
\begin{equation*}
A = \{ (U,s) : U \subseteq Y \text{ open and }s \in \EE(U) \}.
\end{equation*}
Then we can define a restriction operation $\meet$ and an overwrite operation $\join$ on $A$ as follows:
\begin{itemize}
\item $(U,s) \meet (V,t) = (U \cap V, s|_{U \cap V})$;
\item $(U,s) \join (V,t) = (U \cup V, s|_{U-V} \cup t)$.
\end{itemize}
With these operations, $A$ is a (left-handed) noncommutative frame, as introduced in \cite{kcv-frames}. Each noncommutative frame $A$ has a commutative shadow $A/\mathcal{D}$, where $\mathcal{D}$ denotes Green's equivalence relation. In the above example, the commutative shadow agrees with the frame of open subsets of $Y$. Geometrically speaking, $A$ can be thought of as (the set of opens of) a noncommutative space covering the topological space $Y$. In fact, the theory of noncommutative frames was motivated by \cite{llb-covers-general-version}, in which such a noncommutative topology was constructed on the points of the Arithmetic Site by Connes and Consani \cite{connes-consani}. 

Noncommutative frames belong to the theory of skew lattices. In this theory, the $\meet$ and $\join$ operations are no longer required to be commutative, as in the case of lattices. Instead they are idempotent, associative operations satisfying the absorption laws
\begin{equation*}
x \meet (x \join y) = x = x \join (x \meet y) \text{ and }
(x \meet y) \join y = y = (x \join y) \meet y.
\end{equation*}
We will recall some basic results for skew lattices in Section \ref{sec:preliminaries}. For a more detailed discussion, we refer the reader to Leech's survey \cite{jl-survey}. For an overview of the primary results, see \cite{jl-journey}. 

In \cite{nctop}, noncommutative frames were used to define noncommutative generalizations of toposes, by replacing the subobject classifier $\mathbf{\Omega}$ with an internal noncommutative frame $\mathbf{H}$ having the subobject classifier as commutative shadow. While $\mathbf{\Omega}$ has a single top element (corresponding to the statement ``true'' in logic), there are now multiple top elements in $\mathbf{H}$. An example of a noncommutative topos that is not an elementary topos, is the category of complete directed graphs with a 4-coloring of the edges. It is impossible here to give an unambiguous sheafification: it is not enough to force every pair of vertices to have an edge between them, there also has to be a choice of what color this edge should be.

In this paper, we will first study which noncommutative frames can be constructed as above from a pair $(Y,\EE)$, where $Y$ is a topological space and $\EE$ is a sheaf on $Y_f$ such that $\EE(Y_f)\neq\varnothing$. These noncommutative frames will be called \emph{spatial}. After giving two examples of left-handed noncommutative frames that are not spatial, we show that there is an adjunction
\begin{equation*}
\begin{tikzcd}
\sh(\mathsf{Sp}_f)^\op~ \ar[rr, bend right=16pt,"{H}"'] & & ~\mathbf{LNFrm} \ar[ll, bend right=16pt,"{G}"']
\end{tikzcd},
\end{equation*}
where $H$ sends a pair $(Y,\EE)$ to its associated noncommutative frame. This adjunction restricts to a categorical duality between spatial noncommutative frames and the pairs $(Y,\EE)$ such that $Y$ is a sober topological space (i.e.\ each irreducible closed subset has a unique generic point).

The above duality for spatial noncommutative frames is based on previous work of Bauer, Cvetko-Vah, Gehrke, van Gool and Kudryavtseva \cite{kcv-priestley}, in which classical Priestley duality is extended to strongly distributive skew lattices, which are the noncommutative counterparts of distributive lattices. In the commutative world, bounded distributive lattices correspond to Priestley spaces \cite{priestley}, or equivalently spectral spaces in the sense of Hochster \cite{hochster}, while frames correspond to  arbitrary topological spaces (or more generally locales). In this sense, the duality presented here can be seen as a natural generalization of \cite[Theorem 3.7]{kcv-priestley}.

For a topological space $Y$, open sets are by definition completely determined by the points they contain. In terms of the frame $\OO(Y)$ of open subsets of $Y$, this means that for any $u,v \in \OO(Y)$ there is a morphism of frames
\begin{equation*}
p : \OO(Y) \longrightarrow \mathbf{2}.
\end{equation*}
such that $p(u)\neq p(v)$. For a noncommutative frame $A$, we similarly say that two elements $a,b \in A$ can be separated if there exists a morphism of noncommutative frames
\begin{equation*}
q : A \longrightarrow P_{\{1_a,1_b\}}
\end{equation*}
such that $q(a) \neq q(b)$, where $P_{\{1_a,1_b\}}$ is the primitive skew lattice consisting of an element $0$ and two top elements $1_a$ and $1_b$. We will show that $A$ can be embedded in a spatial noncommutative frame if and only if each two distinct elements can be separated. 

In the last part of the paper, we prove a duality result for left-handed noncommutative frames that are not necessarily spatial. In general, the commutative shadow $L = A/\mathcal{D}$ of a noncommutative frame can be seen as the frame of open subsets of a locale $Y$. A familiar object in locale theory is the dissolution locale $Y_d$ of $Y$. By definition, sublocales of $Y$ correspond to closed sublocales of $Y_d$. We will see in Theorem \ref{thm:duality} that there is an equivalence of categories
\begin{equation*}
\mathbf{LNFrm} ~\simeq~ \sh(\mathsf{Loc}_d)^\op
\end{equation*}
where $\mathbf{LNFrm}$ is the category of left-handed noncommutative frames, and $\sh(\mathsf{Loc}_d)$ is a suitable category of pairs $(Y,\FF)$, with $Y$ a locale and $\FF$ a sheaf on $Y_d$, satisfying $\FF(Y_d) \neq \varnothing$. This duality result is again based on Theorem 3.7 in \cite{kcv-priestley}.

In a recent paper \cite{frame-of-nuclei-spatial}, \'Avila, Bezhanishvili, Morandi and Zald\'ivar show that, if $Y$ is a sober topological space, then $Y_f$ is the space of points of the dissolution locale $Y_d$. In other words, $Y$ equipped with the front topology is the topological space that gives the best approximation to the dissolution locale $Y_d$. This allows us in Proposition \ref{prop:comparison} to show how the two duality results (the one for spatial noncommutative frames and the one for general noncommutative frames) relate to each other.

\section{Preliminaries}
\label{sec:preliminaries}

\subsection{Frames and locales}

Following \cite{picado-pultr}, a \emph{frame} is a complete lattice $L$ that satisfies the infinite distributive law:
\begin{equation*}
(\bigvee_{i \in I} x_i) \meet y ~=~ \bigvee_{i \in I} x_i \meet y
\end{equation*}
for $y \in L$ and $x_i \in L$ for all $i \in I$. A \emph{frame homomorphism} $h:L\to M$ between frames $L$ and $M$  is a map $L\to M$ that preserves all joins (including the bottom $0$) and all finite meets (including the top $1$). We denote the resulting category by $\mathbf{Frm}$.

The \emph{category of locales} $\mathbf{Loc}$ is defined as the opposite of the category $\mathbf{Frm}$. The frame associated to a locale $Y$ will be called the \emph{frame of open subsets of $Y$}, and will be denoted by $\OO(Y)$. A morphism of locales $Y \to Y'$ is by definition a frame morphism $\OO(Y') \to \OO(Y)$. Given a topological space $X$,  the lattice of all open subsets of $X$ is a frame, see \cite[Chapter IX]{maclane-moerdijk-sheaves} or \cite[Chapter II]{johnstone-stone} for details. This frame will also be denoted by $\OO(X)$. In this way, we can associate to each topological space a locale, called the \emph{underlying locale}. Any continuous map of topological spaces $f : X \to X'$ induces a morphism of the underlying locales, defined by taking inverse images of open subsets $f^{-1} : \OO(X') \to \OO(X)$. This defines a functor
\begin{equation*}
i ~:~ \mathbf{Sp} \longrightarrow \mathbf{Loc}.
\end{equation*}

Given a locale $Y$, a \emph{point} of $Y$ is a morphism of locales $\mathbf{1} \to Y$, where $\mathbf{1}$ denotes (the underlying locale of) the one-point space. Equivalently, a point is a frame homomorphism $p:\OO(Y)\to \mathbf{2}$. Here $\mathbf{2}$ denotes the frame of open sets of the one-point space: it has two elements $1>0$. We denote the set of points of $Y$ by $\pt(Y) $. A topology is defined on $\pt(Y)$ by letting the opens  be all sets of the form:
\[
U_a =\{p:\OO(Y)\to \mathbf{2},~ p(a)=1\}
\]
for $a \in \OO(Y)$. With this topology, $\pt(Y)$ is called \emph{the space of points of $Y$}. A morphism of locales $Y\to Y'$ defined by $h: \OO(Y') \to \OO(Y)$ induces a continuous map $\mathrm{pt}(h):\pt(Y)\to \pt(Y')$, $\pt(h)(p)=p\circ h$. This makes $\mathrm{pt}$ into a functor from $\mathbf{Loc}$ to $\mathbf{Sp}$. Note that $(\mathrm{pt}(h))^{-1}(U_a)=U_{h(a)}$.
 
\begin{theorem}[{\cite[II, 1.4]{johnstone-stone}}]
The functor $\pt: \mathbf{Loc}\to \mathbf{Sp}$ is right adjoint to the functor $i:\mathbf{Sp}\to \mathbf{Loc}$.
\end{theorem} 

A locale $Y$ is called \emph{spatial} if it is in the essential image of $i$, i.e.\ if it occurs as underlying locale of some topological space. A topological space is called \emph{sober} if it is in the essential image of $\mathrm{pt} : \mathbf{Loc} \to \mathbf{Sp}$, i.e.\ if it is isomorphic to the space of points of some locale. By the general theory of adjunctions, $i$ and $\mathrm{pt}$ induce an equivalence of categories between sober topological spaces and spatial locales.

For $X$ a topological space, the sober topological space $\hat{X} = \pt(i(X))$ will be called the \emph{sobrification of $X$}. From the adjunction above, there is a natural map $X \to \hat{X}$ inducing an isomorphism of underlying locales.

For more on locales and sober topological spaces, we refer to \cite[Chapter IX]{maclane-moerdijk-sheaves} or \cite[Chapter II]{johnstone-stone}.

\subsection{Skew lattices}

A skew lattice is a set $A$ endowed with a pair of idempotent, associative operations  $\land$ and $\lor$ such that the absorption laws
\begin{equation*}
x\land (x\lor y)=x=x\lor (x\land y)\text{ and } (x\land y)\lor y=y=(x\lor y)\land y
\end{equation*}
are satisfied. Given skew lattices $A$ and $A'$, a \emph{homomorphism} of skew lattices is a map $f:A\to A'$ that preserves finite meets and joins, or in other words:
\begin{itemize}
\item $f(a\land b)=f(a)\land f(b)$, for all $a,b\in A$;
\item $f(a\lor b)=f(a)\lor f(b)$, for all $a,b\in A$.
\end{itemize}
The \emph{natural partial order} is defined on any skew lattice $A$ by: $a\leq b$ iff $a\land b=b\land a=a$, or equivalently, $a\lor b=b=b\lor a$. \emph{Green's equivalence relation $\DD$} is defined on $A$ by: $a~\DD~b$ iff $a\land b\land a=a$ and $b\land a\land b=b$, or equivalently, $a\lor b\lor a=a$ and $b\lor a\lor b=b$. We denote the equivalence class of an element $a \in A$ by $\DD_a$ or $[a]$. Leech's First Decomposition Theorem \cite{jl-rings}, states that $\DD$ is a congruence on the skew lattice $A$ and that $A/\DD$ is the maximal lattice image of $A$. We will also refer to $A/\mathcal{D}$ as the \emph{commutative shadow} of $A$. 

A skew lattice is called \emph{left handed}, if it satisfies the identity $x\land y\land x=x\land y,$ or equivalently, $x\lor y\lor x=y\lor x$; it is called \emph{right-handed} if it satisfies the identity $x\land y\land x=y\land x$, or equivalently, $x\lor y\lor x=x\lor y$. By Leech's Second Decomposition Theorem any skew lattice factors as a pullback of a left handed skew lattice by a right-handed skew lattice over their common maximal lattice image, see \cite{jl-rings}.

A skew lattice is said to be \emph{strongly distributive} if it satisfies the identities:
\[
(x\lor y)\land z = (x\land z)\lor (y\land z) \text{ and } x\land (y\lor z)=(x\land y)\lor (x\land z).
\]
Let $a$ be an element of a strongly distributive skew lattice $A$. If $u$ is a $\mathcal{D}$-class with $u \leq a$, then there exists a unique element $b \in A$ such that $b \leq a$ and $[b]=u$. We will call $b$ the \emph{restriction of $a$ to $u$}.

As shown by Leech in \cite{jl-normal}, a skew lattice is strongly distributive if and only if it is symmetric, distributive and normal, where a skew lattice $A$ is called:
\begin{itemize}
 \item \emph{symmetric} if for any $x,y\in A$, $x\lor y=y\lor x$ is equivalent to $x\land y=y\land x$;
 \item \emph{distributive} if it satisfies  the identities:
\begin{eqnarray*}
x\land (y\lor z)\land x=(x\land y\land x)\lor (x\land z\land x)\\
x\lor (y\land z)\lor x=(x\lor y\lor x)\land (x\lor z\lor x);
\end{eqnarray*}
\item \emph{normal} if it satisfies the identity $x\land y\land z\land x=x\land z\land y\land x$.
\end{itemize}
A skew lattice $A$ is normal if and only if given any $a\in A$ the set
\[
a\!\downarrow~ =\{u\in A\,|\, u\leq a\}
\]
is a lattice, see \cite{jl-normal}. Normal skew lattices are sometimes called \emph{local lattices}. Finally, a \emph{skew lattice with $0$} is a skew lattice with a distinguished element $0$ satisfying $x\lor 0=x=0\lor x$, or equivalently, $x\land 0=0=0\land x$.

\begin{example}\label{ex:partial-func}
Let $R,S$ be non-empty sets and denote by $\PP(R,S)$ the set of all partial functions from $R$ to $S$. We define the following operations on $\PP(R,S)$:
\begin{eqnarray*}
f \land g & = &\restricted{f}{\dom(f) \cap \dom(g)}\\
f \lor g& = &g\cup \restricted{f}{\dom(f) \setminus \dom(g)}
\end{eqnarray*}
In \cite{jl-normal}, Leech showed that $(\PP(R,S);\land, \lor)$ is a strongly distributive left handed skew lattice with $0$. Moreover, given $f,g\in (\PP(R,S);\land, \lor)$ the following hold:
\begin{itemize}
\item $f\, \DD\, g$ iff $\dom(f)=\dom(g)$;
\item $f\leq g$ iff $f=\restricted{g}{\dom(f) \cap \dom(g)}$;
\item $\PP(R,S)/\DD\cong \PP(R).$
\end{itemize}
We will see in the next subsection that $\PP(R,S)$ is a noncommutative frame.
\end{example}

\subsection{Noncommutative frames}  \label{ssec:ncframes}

A subset $\{x_i : i\in I\} \subseteq A$ of a skew lattice $A$ is said to be a \emph{commuting subset} if it is nonempty and moreover $x_i\land x_j=x_j\land x_i$ and $x_i\lor x_j=x_j\lor x_i$ for all $i,j\in I$. We say that a skew lattice is \emph{join complete} if all commuting subsets have suprema with respect to the natural partial ordering. Leech showed in \cite{jl-boolean} that a join complete skew lattice $A$ always has a maximal $\mathcal D$-class.

A \emph{noncommutative frame} is a strongly distributive, join complete skew lattice $A$ with $0$ that satisfies the infinite distributive laws
\begin{equation}\label{eq:inf-dist-law}
 (\bigvee_i x_i)\land y=\bigvee_i (x_i\land y) \qquad \text{and} \qquad
x\land ({\bigvee_i y_i})=\bigvee_i (x\land y_i)
\end{equation}
for all $x,y\in A$ and all commuting subsets $\{x_i : i\in I\}$, $\{y_i : i\in I\}\subseteq A$.  

By a result of \cite{jl-discriminator}, any join complete, normal skew lattice $A$ with $0$ (for instance, any noncommutative frame) satisfies the following:
\begin{itemize}
\item any nonempty commuting subset $C\subseteq A$ has an infimum w.r.t.\ the natural partial order, to be denoted by $\bigwedge C$;
\item any nonempty subset $C\subseteq A$ has an infimum w.r.t.\ the natural partial order, to be denoted by $\bigcap C$ (or  by $x\cap y$ in the case $C=\{x,y\}$);
\item if $C$ is a nonempty commuting subset of $A$, then $\bigwedge C=\bigcap C$.
\end{itemize}
We call $\bigcap C$ the \emph{intersection} of $C$.

The following is an example of a noncommutative frame (it is easy to check that it satisfies all necessary properties).

\begin{proposition} \label{prop:partial-functions-is-ncframe}
The set $\PP(R,S)$ of all partial functions from $R$ to $S$ with the operations $\meet$, $\join$ defined as in Example \ref{ex:partial-func} is a noncommutative frame.
\end{proposition}

If $\{x_i : i \in I \} \subseteq A$ is a commuting subset, and $h : A \to A'$ is a homomorphism of skew lattices, then $\{ h(x_i) : i \in I \}$ is a commuting subset of $A'$. For $A$ and $A'$ noncommutative frames, we say that $h : A \to A'$ is a \emph{morphism of noncommutative frames} if it satisfies the following properties:
\begin{itemize}
\item $h$ is a homomorphism of skew lattices;
\item $h(0)=0$;
\item if $t$ is in the maximal $\mathcal{D}$-class of $A$, then $h(t)$ is in the maximal $\mathcal{D}$-class of $A'$;
\item $h(\bigvee_i x_i)=\bigvee _i h(x_i)$, for all commuting subsets $\{x_i : i\in I\}$.
\end{itemize}
The category of noncommutative frames will be denoted by $\mathbf{NFrm}$, and the full subcategory of left handed noncommutative frames by $\mathbf{LNFrm}$.

Note that a morphism $h: A \to A'$ of noncommutative frames are compatible with the congruence $\DD$, in other words $a~\DD~b$ implies $h(a)~\DD~h(b)$. In fact, the induced map ${h}_c:A/\DD\to A'/\DD$, $[x] \mapsto [h(x)]$ is a frame morphism. Moreover, given a noncommutative frame, the natural projection $\pi_A:A\to A/\DD$, mapping $x$ to its $\DD$-class $[x]$, is always a morphism of noncommutative frames.

\section{Duality for spatial noncommutative frames}
\label{sec:duality-spatial-ncframe}

\subsection{The front topology}

Let $Y$ be a topological space. We say that a subset $S \subseteq Y$ is \emph{locally closed} if it can be written as $S = U \cap V$ with $U$ open and $V$ closed. The locally closed sets are the basis of a topology, that we will call the \emph{front topology}. This is the terminology of Simmons \cite{simmons} that is also used in the recent paper by \'{A}vila, Bezhanishvili, Morandi and Zald\'{i}var \cite{frame-of-nuclei-spatial}. We write $Y_f$ for $Y$ equipped with the front topology.

\begin{example} \ 
\begin{enumerate}
\item If $Y$ is Hausdorff, then all points in $Y$ are closed. As a result, $Y_f$ has the discrete topology.
\item Take $Y = \Spec(\mathbb{Z})$ with the Zariski topology. Then $Y_f$ is homeomorphic to 
\[
\{0\} \cup \{ \frac{1}{n} : n \in \mathbb{N}_{>0} \} ~\subseteq~ \mathbb{R}
\]
with the usual (Euclidean) topology.
\end{enumerate}
\end{example}

\subsection{Terminology for sheaves}
\label{ssec:sheaves}
We recall the main notions of sheaf theory, and introduce the necessary terminology. For details, we refer to \cite{maclane-moerdijk-sheaves}. For $Y$ a topological space, a \emph{presheaf on $Y$} is a functor
\[
\EE : \OO(Y)^\op \longrightarrow \mathbf{Sets}
\]
to the category of sets. For $U \in \OO(Y)$, the elements of $\EE(U)$ will be called the \emph{(local) sections over $U$}. For $s \in \EE(U)$, we also say that \emph{$U$ is the domain of $s$} and we write $U = \dom(s)$. A family of elements $(s_i)_{i \in I}$ with $s_i \in \EE(U_i)$ is called a \emph{matching family} if for all $i,j \in I$ we have
\begin{equation*}
s_i|_{U_i \cap U_j} = s_j|_{U_i \cap U_j}.
\end{equation*}
Now $\EE$ is a \emph{sheaf on $Y$} if and only if for every matching family $(s_i)_{i \in I}$ as above, there is a unique section $s \in \EE(U)$ for $U = \bigcup_{i \in I} U_i$, such that $s|_{U_i}  = s_i$ for all $i \in I$.

For an element $p \in Y$, we define the \emph{stalk} $\EE_p$ as the filtered colimit
\begin{equation*}
\EE_p ~=~ \varinjlim_{U \ni p} \EE(U).
\end{equation*}
So every section $s \in \EE(U)$ with $U \ni p$ determines an element of $\EE_p$, that we will call the \emph{germ of $s$ at $p$}, and we write it as $\mathrm{germ}_p(s)$. For $s \in \EE(U)$ and $s' \in \EE(U')$ with $p \in U \cap U'$ we have $\mathrm{germ}_p(s) = \mathrm{germ}_p(s')$ if and only if there is some $p \in V \subseteq U \cap U'$ such that $s|_V = s'|_V$.

The \emph{\'etale space} $E$ of $\EE$ is (as a set) the disjoint union
\begin{equation*}
E = \bigsqcup_{p \in Y} \EE_p.
\end{equation*}
There is a projection map $\pi : E \to Y$ sending each element of $\EE_p$ to $p$.
Each section $s \in \EE(U)$ defines a function $s : U \to E$ given by 
\begin{equation} \label{eq:section-as-function}
s(p) = \mathrm{germ}_p(s)
\end{equation}
and this function satisfies $\pi(s(p)) = p$ for all $p \in U$. We define a topology on $E$ by taking as subbasis the set of all sets of the form $s(U)$ for $U \subseteq Y$ an open subset and $s \in \EE(U)$. Then $\pi$ is a local homeomorphism, each function $s : U \to E$ defined by $(\ref{eq:section-as-function})$ is continuous, and moreover:
\begin{proposition}
Let $\EE$ be a sheaf on $Y$ and let $U \subseteq Y$ be an open subset. For every continuous function $f : U \to E$ such that $\pi(f(p)) = p$ for all $p \in U$, there is a unique $s \in \EE(U)$ such that
\begin{equation*}
f(p) = \mathrm{germ}_p(s).
\end{equation*}
\end{proposition}
In this way, we can recover $\EE$ from its \'etale space. Further, each local homeomorphism $\pi : E \to Y$ is the \'etale space of a sheaf $\EE$. This equivalence between sheaves and \'etale spaces is functorial: if $\phi : \EE \to \EE'$ is a morphism of sheaves over $Y$, then
\begin{equation*}
\psi : E \longrightarrow E',\quad \psi(\mathrm{germ}_p(s)) = \mathrm{germ}_p(\phi(s))
\end{equation*}
is a well-defined continuous map between the respective \'etale spaces, such that $\pi' \circ \psi = \pi$. Conversely, every continuous $\psi : E \to E'$ such that $\pi' \circ \psi = \pi$ is of this form for a unique $\phi$.

The definitions for presheaves, sections, matching families and sheaves extend word for word to the more general case where $Y$ is a locale. For every point $p \in \pt(Y)$, we can define the stalk $\EE_p = \varinjlim_{U \ni p} \EE(U)$ of a (pre)sheaf $\EE$ and the germ of a section $s \in \EE(U)$ at $p \in U$ (if $p$ is interpreted as a frame morphism $p : \OO(Y) \longrightarrow \mathbf{2}$, then the condition $p \in U$ means $p(U)=1$). The \'etale space $E$ is then a topological space with a projection map
\begin{equation*}
\pi : E \longrightarrow \pt(Y).
\end{equation*}
The main difference with the case of topological spaces, is that we cannot recover the sheaf $\EE$ from its \'etale space $E$. So here sheaves are more general than \'etale spaces (one can remedy this by defining local homeomorphisms of locales, but we will not follow this approach here).

\subsection{Spatial noncommutative frames}
\label{ssec:spatial-definition}

Let $\EE$ be a sheaf on $Y_f$ such that $\EE(Y_f) \neq \varnothing$. Recall from \cite{nctop} that we can then construct a noncommutative frame $H(Y,\EE)$ as follows. The elements are pairs $(U,s)$ with $U \subseteq Y$ open (for the original topology on $Y$) and $s \in \EE(U)$. The meet and join operations are defined as follows:
\begin{itemize}
\item $(U,s) \meet (V,t) = (U \cap V, s|_{U \cap V})$,
\item $(U,s) \join (V,t) = (U \cup V, s|_{U-V} \cup t)$,
\end{itemize}
where $s|_{U-V} \cup t$ is the unique section restricting to $s|_{U-V}$ on $U-V$ and to $t$ on $V$. It is easy to verify that $H(Y,\EE)$ is a noncommutative frame. In fact, $H(Y,\EE)$ can be seen as a subset of the noncommutative frame of all partial functions from $X$ to 
\begin{equation*}
E=\bigsqcup_{p\in Y} \EE_p=\{\mathrm{germ}_p(s)~:~ p\in Y,~ s\in \EE(U) ~\text{for }U \ni p \},
\end{equation*}
by identifying $(U,s)$ with the function 
\begin{equation*}
s : U \longrightarrow E,\quad s(p) = \mathrm{germ}_p(s).
\end{equation*}
It therefore suffices to show that $H(Y,\EE)$  is closed under the meet and join operations, which is the case exactly because we work with the front topology.

Observe that given any global section $t$, each downset of the form $(Y,t) {\downarrow}$ is isomorphic to the frame $\OO(Y)$.  Moreover,
$(U,s)~\mathcal{D} ~(U',s')$ if and only if $U=U'$.
It follows that $H(Y,\EE)$  has a unique bottom element  $0 = (\emptyset, \emptyset)$, and it  has a top $\DD$-class  $T=\{ (Y,t) : t \in \EE(Y_f) \}$, in other words top elements correspond to global sections.
The factor algebra $H(Y,\EE) / \mathcal{D}$ is isomorphic to the frame $\OO(Y)$. We  view $H(Y,\EE)$ as the set of opens of a noncommutative topological space with commutative shadow $Y$.

\begin{definition}[Spatial noncommutative frames]
A left-handed noncommutative frame $A$ will be called \emph{spatial} if there exists a topological space $Y$ and a sheaf $\EE$ on $Y_f$ with $\EE(Y_f) \neq \varnothing$ such that $A \cong H(Y,\EE)$ as above.
\end{definition}

If $A$ is a frame, i.e.\ $A = A/\mathcal{D}$, then this coincides with the usual definition: $A$ is a spatial frame if and only if $A = \OO(Y)$ for some topological space $Y$.

Take a topological space $Y$ and a sheaf $\EE$ on $Y_f$. The natural map $Y \to \hat{Y}$ from $Y$ to its sobrification induces a map $\beta : Y_f \to (\hat{Y})_f$. It can now be checked that $H(Y,\EE) \cong H(\hat{Y},\beta_*\EE)$. So in the above definition, we can assume that $Y$ is sober.

\subsection{Examples of noncommutative frames that are not spatial}
\label{ssec:not-spatial}

Let $A$ be a noncommutative frame such that its commutative shadow $A/\mathcal{D}$ is not spatial. Then $A$ can not be spatial either. However, in this subsection we would like to give two examples of left-handed noncommutative frames $A$ that are not spatial, despite having a spatial commutative shadow. 

Consider the set $B$ of pairs $(U,f)$ with $U \subseteq \mathbb{R}$ an open set for the usual (Euclidean) topology, and $f : U \to \{0,1\}$ an arbitrary function. We define an equivalence relation $\sim$ such that $(U,f) \sim (V,g)$ if and only if $V = U$ and moreover $\{x \in U : f(x) \neq g(x)\}$ is countable (by countable we always mean either finite or countably infinite). We now define a noncommutative frame $A'$ with as elements the equivalence classes
\begin{equation*}
A' = B/\!\sim
\end{equation*}
and with meet and join defined by
\begin{equation*}
\begin{split}
&(U,f) \meet (V,g) = (U\cap V,f|_{U \cap V}) \\
&(U,f) \join (V,g) = (U\cup V,f|_{U - V} \cup g|_V),
\end{split}
\end{equation*}
where $h = f|_{U - V} \cup g|_V$ is the function defined by $h(x) = f(x)$ for $x \in U-V$ and $h(x) = g(x)$ for $x \in V$. Note that meet and join do not depend on the chosen representatives. It is straightforward to check that $A'$ is a strongly distributive skew lattice with $0$, with $A'/\mathcal{D} \cong \OO(\mathbb{R})$. From \cite[Theorem 5.1]{completeness-issues} it follows that $A'$ is a noncommutative frame if we show that it is join complete. So take a commuting family of elements $(U_i,f_i)$ indexed by $i \in I$. Then $f_i|_{U_i \cap U_j}$ and $f_j|_{U_i \cap U_j}$ disagree on only countably many points. Let $U = \bigcup_{i \in I} U_i$. Since $\mathbb{R}$ is strongly Lindel\"of, we can find an countable subset $J \subseteq I$ such that $U = \bigcup_{j \in J} U_j$. For each $p \in U$, take a $j(p) \in J$ such that $p \in U_{j(p)}$. Then define:
\begin{equation*}
f: U \to \{0,1\},\quad f(p) = f_{j(p)}(p).
\end{equation*}
We claim that $(U,f) = \bigvee_{i \in I} (U_i,f_i)$. Consider the set
\begin{equation*}
S_i = \{ p \in U_i : f(p) \neq f_i(p) \}.
\end{equation*}
We have to show that $S_i$ is countable. This follows from:
\begin{equation*}
S_i \subseteq \bigcup_{j \in J} \{ p \in U_i \cap U_j : f_i(p) \neq f_j(p) \}
\end{equation*}
(the right hand side is countable, being a countable union of countable sets). Now let $g : U \to \{0,1\}$ be any other function such that $g|_{U_i}$ and $f_i$ agree outside of a countable set. Consider the set
\begin{equation*}
S' = \{ p \in U : f(p) \neq g(p) \}.
\end{equation*}
Then
\begin{equation*}
S' \subseteq \bigcup_{j \in J} \{ p \in U_j : f_j(p) \neq g(p) \}
\end{equation*}
so $S'$ is countable. This shows that $(U,f)$ is well-defined and $(U,f) = \bigvee_{i \in I} (U_i,f_i)$. It follows that $A'$ is a noncommutative frame. We claim that $A'$ is not spatial, but we will postpone the proof to Subsection \ref{ssec:proof-not-spatial}.

As another example, consider the set $A''$ of pairs $(U,f)$ with $U \subseteq \mathbb{R}$ an open subset, and $f : U \to \{0,1\}$ a function such that $\{x \in U : f(x) = 1 \}$ is countable. Again, we define meet and join as
\begin{equation*}
\begin{split}
&(U,f) \meet (V,g) = (U\cap V,f|_{U \cap V}) \\
&(U,f) \join (V,g) = (U\cup V,f|_{U - V} \cup g|_V).
\end{split}
\end{equation*}
In an analogous way as in the previous example, we can show that $A''$ is a join complete strongly distributive skew lattice with commutative shadow $A''/\mathcal{D} \cong \OO(\mathbb{R})$. So by \cite[Theorem 5.1]{completeness-issues} it is a noncommutative frame. Again, we claim that $A''$ is not spatial, but we postpone the proof to Subsection \ref{ssec:proof-not-spatial}.

\subsection{$H$ as a functor} \label{ssec:H-as-functor}

We can interpret the map $(Y,\EE) \mapsto H(Y,\EE)$ from Subsection \ref{ssec:spatial-definition} as a functor $H$, in the following way. We define the category $\sh(\mathsf{Sp}_f)$ of \emph{sheaves for the front topology} as the category with
\begin{itemize}
\item as objects the pairs $(Y,\EE)$ where $Y$ is a topological space and $\EE$ is a sheaf on $Y_f$ such that $\EE(Y_f) \neq \varnothing$;
\item as morphisms the pairs $(f,\lambda) : (Y,\EE) \to (Y',\EE')$ where $f : Y \to Y'$ is a continuous map and $\lambda : \EE' \to f_*\EE$ is a sheaf morphism (note that $f$ also defines a continuous map $Y_f \to (Y')_f$).
\end{itemize}
Now $H$ becomes a functor
\begin{equation*}
H : \sh(\mathsf{Sp}_f)^\op \longrightarrow \mathbf{LNFrm}
\end{equation*}
by associating to $(f,\lambda)$ the morphism of noncommutative frames
\begin{equation*}
H(f,\lambda) : H(Y',\EE') \longrightarrow H(Y,\EE)
\end{equation*}
sending $(U,s) \in H(Y',\EE')$ to $(f^{-1}(U),\lambda(s)) \in H(Y,\EE)$.

\subsection{Primitive quotients} \label{ssec:primitive-quotients} We would like to construct a left adjoint $G$ to the functor $H$ above. First, we recall the notion of primitive quotient from \cite{kcv-priestley}.

\begin{definition}[{\cite{kcv-priestley}}] \label{def:equivalence-rel}
Let $A$ be a left-handed strongly distributive skew lattice and let $p : A \to \mathbf{2}$ be a lattice homomorphism. For $a,b \in A$ with $p(a) = p(b) = 1$, we define
\begin{equation*}
a \sim_p b \quad\Leftrightarrow\quad \exists c,d \in A,~ p(c)=0,~p(d)=1,~ (a \meet d) \join c = (b \meet d) \join c.
\end{equation*}
If $a \sim_p b$, then we say that $a$ and $b$ \emph{agree in $p$}.
\end{definition}

Recall from \cite[Subsection 6.2]{kcv-priestley} that $\sim_p$ is an equivalence relation. The equivalence class of an element $a$ is written as $[a]_{\sim_p}$. By \cite[Proposition 6.1]{kcv-priestley}, the quotient $A/\!\!\sim_p$ is a skew lattice with a unique nontrivial $\mathcal{D}$-class. The quotient map is given by
\begin{gather*}
\pi : A \to A/\!\!\sim_p \\
\pi(a) = \begin{cases}
0 & \text{if }p(a) = 0, \\
[a]_{\sim_p} & \text{if }p(a) = 1.
\end{cases}
\end{gather*}
It is easy to check that $\pi$ is a morphism of noncommutative frames if and only if $p$ is.

A skew lattice that only has one nontrivial $\mathcal{D}$-class, is called \emph{primitive}. Primitive left-handed skew lattices are denoted by $P_T$, where $T$ is the set of top elements. Any morphism $q : A \to P_T$ to a primitive skew lattice such that $q/\mathcal{D} = p$ factors as $q = t \circ \pi$ for some $t : A/\!\!\sim_p~ \to P_T$, see \cite[Proposition 6.1]{kcv-priestley}.

For spatial noncommutative frames $A = H(Y,\EE)$, the equivalence relation can be phrased in terms of the stalks of $\EE$.

\begin{lemma} \label{lmm:stalks-vs-equivalence-rel}
Take $(U,s),(U',s') \in H(Y,\EE)$ and $p \in U \cap U'$. Then 
\begin{equation*}
(U,s) \sim_p (U',s')
\end{equation*}
if and only if $\mathrm{germ}_p(s) = \mathrm{germ}_p(s')$.
\end{lemma}
\begin{proof}
This is analogous to the proof of \cite[Lemma 5.1]{kcv-priestley}.
%
%
\end{proof}

\subsection{\'Etale space associated to a noncommutative frame}
\label{ssec:etale-space}

Let $A$ be a left-handed noncommutative frame. We would like to construct an object $G(A)$ in $\sh(\mathsf{Sp}_f)$ such that $H(G(A))$ is the spatial noncommutative frame that ``best approximates $A$'', or more precisely such that $G$ is left adjoint to $H$ as functors
\begin{equation*}
\begin{tikzcd}
\sh(\mathsf{Sp}_f)^\op~ \ar[rr, bend right=16pt,"{H}"'] & & ~\mathbf{LNFrm} \ar[ll, bend right=16pt,"{G}"']
\end{tikzcd}.
\end{equation*}
So we want to show that there is a natural bijection between noncommutative frame morphisms $A \to H(Y,\EE)$ and morphisms $(Y,\EE) \to G(A)$ in $\sh(\mathsf{Sp}_f)$.

We write $G(A) = (Y_A,\EE_A)$ where $Y_A$ is a topological space and $\EE_A$ is a sheaf on $Y_{A,f}$. We then define:
\begin{equation*}
Y_A = \pt(A/\mathcal{D})
\end{equation*}
to be the space of points of $A/\mathcal{D}$. We write $Y_{A,f}$ for $Y_A$ with its front topology. We can now use the equivalence relations $\sim_p$ from Definition \ref{def:equivalence-rel} to construct a space $E_A$ and a local homeomorphism $\pi_A : E_A \to Y_{A,f}$. We define
\begin{equation*}
E_p = \{ [a]_{\sim_p} : a\in A,~ p(a) = 1 \}
\end{equation*}
for each $p \in Y_A$, and
\begin{equation*}
E_A = \bigsqcup_{p \in Y_A} E_p = \{ (p,[a]_{\sim_p}) : p \in Y_A,~ a \in A,~ p(a) = 1  \}.
\end{equation*}
We can then define $\pi_A : E_A \longrightarrow Y_{A,f}$ as $\pi_A(p,[a]_{\sim_p}) = p$. The topology on $E_A$ is defined as follows. For $a \in A$, its $\mathcal{D}$-class $[a]$ corresponds to the open set
\begin{equation*}
U_a = \{ p \in Y_A : p(a) = 1 \} \subseteq Y_A.
\end{equation*}
Note that the set $U_a$ is in particular open in $Y_{A,f}$. We write $U_{a,f}$ for $U_a$ with the front topology. Each $a \in A$ defines a function
\begin{equation*}
s_a : U_a \longrightarrow E_A,\quad s_a(p) = (p,[a]_{\sim_p})
\end{equation*}
satisfying $\pi_A(s_a(p)) = p$. Then the subsets of the form
\begin{equation*}
s_a(Z) = \{ (p,[a]_{\sim_p}) : p \in Z \} \subseteq E_A
\end{equation*}
for $a \in A$ and $Z \subseteq U_a$ locally closed, generate a topology on $E_A$.

\begin{theorem}
With the notations above, $s_a : U_{a,f} \longrightarrow E_A$ is continuous for all $a \in A$, and the map $\pi_A : E_A \longrightarrow Y_{A,f}$ is a local homeomorphism.
\end{theorem}
\begin{proof}
Take $W=s_b(Z) \subseteq E_A$ for $b \in A$ and $Z \subseteq U_b$ locally closed. To show that $s_a : U_{a,f} \to E_A$ is continuous, it is enough to show that $s_a^{-1}(W)$ is open in $Y_f$. Note that 
\begin{equation*}
s_a^{-1}(W)= \{p\in Z \cap U_a \,|\, a \sim_p b \}.
\end{equation*}
Take any $p\in s_a^{-1}(W)$. There exist $c,d\in A$ such that $p(c)=0$, $p(d)=1$ and $(a \meet d)\join c = (b\meet d)\join c$. Let $U_p=Z \cap U_a \cap U_d - U_c$, and note that $p\in U_p$ and that $U_p \subseteq Y_f$ is open. We claim that $U_p\subseteq s_a^{-1}(W)$. To see this, take any $p'\in U_p$. Then $p'(a)=p'(b)=p'(d)=1$ and $p'(c)=0$, so that $a\sim_{p'} b$ follows. Thus $p'\in s_a^{-1}(W)$, which proves $U_p\subseteq s_a^{-1}(W)$. 

To prove that $\pi_A$ is a local homeomorphism, note that 
\begin{equation*}
\pi_A|_{s_a(U_a)}: s_a(U_a) \to U_a
\end{equation*}
is a bijection with a continuous inverse $s_a$.
\end{proof}

\begin{remark}
If $Y_A$ is Hausdorff, then $Y_{A,f}$ is discrete. Since 
\begin{equation*}
\pi_A : E_A \longrightarrow Y_{A,f} 
\end{equation*}
is a local homeomorphism, this implies that $E_A$ is discrete as well. 
\end{remark}

\subsection{Adjunction between $G$ and $H$} \label{ssec:adjunction}

Let $U \subseteq Y_{A,f}$ be an open subset and take a continuous map $s : U \longrightarrow E_A$ such that $\pi(s(u)) = u$ for all $u \in U$. Fix a point $p \in Y_{A,f}$. Take an $a \in A$ such that $s(p) = (p,[a]_{\sim_p})$. Then
\begin{equation*}
V = s^{-1}(s_a(U_a)) = \{ p \in U : s(p) = (p,[a]_{\sim p}) \}
\end{equation*}
is an open subset of $Y_{A,f}$, containing $p$, such that
\begin{equation*}
s|_V = s_a|_V.
\end{equation*}
So, locally, $s$ is of the form $s_a$ for some $a \in A$.

The sheaf $\EE_A$ on $Y_{A,f}$ associated to $\pi_A : E_A \to Y_{A,f}$ is given by
\begin{equation*}
\EE_A(U) = \{ s : U \to E_A \text{ continuous with }\pi_A(s(p)) = p \text{ for all } p \in U \}
\end{equation*}
for $U \subseteq Y_{A,f}$ open. If $A$ and $A'$ are two left-handed noncommutative frames, and $\phi : A \to A'$ is a morphism, then there is an induced continuous map
\begin{equation*}
f : Y_{A'} \longrightarrow Y_A,\quad f(p) = p \circ \phi
\end{equation*}
and a morphism of sheaves
\begin{equation*}
\lambda : \EE \longrightarrow f_*\EE',\quad \lambda(s_a) = s_{\phi(a)}.
\end{equation*}
This uniquely determines $\lambda$ since each $s \in \EE_A(U)$ is locally of the form $s_a$ for $a \in A$.

\begin{theorem}
In the diagram
\begin{equation*}
\begin{tikzcd}
\sh(\mathsf{Sp}_f)^\op~ \ar[rr, bend right=16pt,"{H}"'] & & ~\mathbf{LNFrm} \ar[ll, bend right=16pt,"{G}"']
\end{tikzcd},
\end{equation*}
the functor $G$ is left adjoint to $H$.
\end{theorem}
\begin{proof}
For $A$ a left-handed noncommutative frame, we write $G(A) = (Y_A,\EE_A)$ as above. Take $(Y,\EE)$ in $\sh(\mathsf{Sp}_f)$ and a morphism of noncommutative frames
\begin{equation*}
\phi : A \longrightarrow H(Y,\EE).
\end{equation*}
This induces a frame morphism $\phi/\mathcal{D} : A/\mathcal{D} \longrightarrow \OO(Y)$, or equivalently, a morphism of locales $f: Y \longrightarrow Y_A$. Moreover, $\phi$ induces a morphism of sheaves
\begin{equation*}
\lambda : \EE_A \longrightarrow f_* \EE,
\end{equation*}
uniquely defined by $\lambda(s_a) = \phi(a)$ for all $a \in A$ (the $\mathcal{D}$-class of $\phi(a)$ is $f^{-1}(U_a)$, so $\phi(a)$ can be seen as an element of $(f_*\EE)(U_a)$). Together, $f$ and $\lambda$ determine a morphism
\begin{equation*}
(f,\lambda) : (Y,\EE) \longrightarrow G(A)
\end{equation*}
in $\sh(\mathsf{Sp}_f)$. Conversely, if such a morphism $(f,\lambda)$ is given, then we can reconstruct $\phi$ using $\phi(a) = \lambda(s_a)$. So there is a bijective correspondence between noncommutative frame morphisms $A \longrightarrow H(Y,\EE)$ and morphisms $(Y,\EE) \longrightarrow G(A)$ in $\sh(\mathsf{Sp}_f)$. It is straightforward to check that this bijection is natural in $A$ and $(Y,\EE)$. So $G$ is left adjoint to $H$.
\end{proof}

\subsection{Proof that the two examples from \ref{ssec:not-spatial} are not spatial}
\label{ssec:proof-not-spatial}

Consider the noncommutative frames $A'$ and $A''$ from Subsection \ref{ssec:not-spatial}. To show that $A'$ and $A''$ are not spatial, it is enough to show that the natural morphisms $A' \to H(G(A'))$ resp.\ $A'' \to H(G(A''))$ are not bijective, by the general theory of adjunctions.

We first compute $G(A') = (Y',\EE')$. Here $Y' = \mathbb{R}$ (with the Euclidean topology), so $Y'_f$ is given by $\mathbb{R}$ with the discrete topology. Take two elements $(U,f)$ and $(V,g)$ in $A'$, and a point $p \in U \cap V$. We claim that $(U,f)\sim_p (V,g)$. Take an open set (for the Euclidean topology) $W \subseteq U \cap V$ with $W \ni p$. Then $W' = W - \{p\}$ is open as well. Take arbitrary elements $(W,h)$ and $(W',h')$ in $A'$. Then:
\begin{equation*}
((U,f) \meet (W,h)) \join (W',h') = ((V,g) \meet (W,h)) \join (W',h')
\end{equation*}
because the corresponding partial functions disagree in at most one point. So $(U,f) \sim_p (V,g)$. Because $(U,f)$ and $(V,g)$ were arbitrary, this means that $\EE_p$ is a singleton for all $p \in Y'$. It follows that $\EE'(U)$ is a singleton for any subset $U \subseteq \mathbb{R}$. So $H(G(A')) = \PP(\mathbb{R})$, which shows that the morphism $A' \to H(G(A'))$ is surjective but not injective. In particular, $A'$ is not spatial.

We now compute $G(A'') = (Y'',\EE'')$. Again $Y'' = \mathbb{R}$ with the Euclidean topology, and $Y''_f$ has the discrete topology. Take two elements $(U,f)$ and $(V,g)$ in $A''$ and an element $p \in U \cap V$. Then it is easy to show that $(U,f) \sim_p (V,g)$ if and only if $f(p) = g(p)$. This shows that $\EE''_p = \{0,1\}$ for all $p \in Y''$. It follows that $H(G(A'')) = \PP(\mathbb{R},\{0,1\})$. The natural morphism
\begin{equation*}
A'' \longrightarrow H(G(A'')) = \PP(\mathbb{R},\{0,1\})
\end{equation*}
sending $(U,f) \in A''$ to $(U,f) \in \PP(\mathbb{R},\{0,1\})$ is injective but not surjective. This shows that $A''$ is not spatial either.

\subsection{Duality for spatial noncommutative frames}

Consider the adjunction 
\begin{equation*}
\begin{tikzcd}
\sh(\mathsf{Sp}_f)^\op~ \ar[rr, bend right=16pt,"{H}"'] & & ~\mathbf{LNFrm} \ar[ll, bend right=16pt,"{G}"']
\end{tikzcd}.
\end{equation*}
Let $\mathbf{D} \subseteq \mathbf{LNFrm}$ be the full subcategory of spatial noncommutative frames, and conversely let $\mathbf{C} \subseteq \sh(\mathsf{Sp}_f)$ be the full subcategory of pairs $(Y,\EE)$ that can be written as $(Y,\EE) \cong G(A)$ for some $A$ in $\mathbf{LNFrm}$. Recall that since $H$ and $G$ are adjoint, they restrict to an equivalence of categories
\begin{equation*}
\begin{tikzcd}
\mathbf{C}^\op~ \ar[rr, bend right=16pt,"{H}"'] & & ~\mathbf{D} \ar[ll, bend right=16pt,"{G}"'].
\end{tikzcd}
\end{equation*}
A pair $(Y,\EE)$ is in $\mathbf{C}$ if and only if the natural morphism 
\begin{equation} \label{eq:sober}
\begin{split}
&(f,\lambda) : (Y,\EE) \longrightarrow G(H(Y,\EE)),\quad\text{given by} \\
&f : Y \longrightarrow \pt(\OO(Y)),\quad p \mapsto \hat{p}\quad\text{with}\quad \hat{p}(U) = \begin{cases}
1\qquad\text{if }p \in U, \\
0\qquad \text{if }p \notin U
\end{cases} \\
&\lambda : \EE_{H(Y,\EE)} \longrightarrow f_*\EE,\quad s_a \mapsto a
\end{split}
\end{equation}
is an isomorphism. In this case, $(Y,\EE)$ will be called a \emph{sober} sheaf for the front topology. We can then state the following duality theorem.

\begin{theorem}[Duality for spatial noncommutative frames] \label{thm:duality-spatial}
The functors $G$ and $H$ induce a dual equivalence between the category $\mathbf{C}$ of sober sheaves for the front topology, and the category $\mathbf{D}$ of spatial left-handed noncommutative frames.
\end{theorem}

We have the following characterization of sober sheaves:

\begin{proposition}
Take $(Y,\EE)$ in $\sh(\mathsf{Sp}_f)$. Then $(Y,\EE)$ is sober if and only if $Y$ is sober.
\end{proposition}
\begin{proof}
If $(Y,\EE)$ is sober, then in particular $f: Y \longrightarrow \pt(\OO(Y))$ is an isomorphism, so $Y$ is sober. 

Conversely, if $Y$ is sober, the component $f: Y \longrightarrow \pt(\OO(Y))$ from (\ref{eq:sober}) is an isomorphism. We would like to show that $\lambda$ is an isomorphism. It is enough to show that the induced morphism on stalks
\begin{equation*}
\lambda_p : (\EE_{H(Y,\EE)})_p \longrightarrow (f_*\EE)_p = \EE_p
\end{equation*}
is an isomorphism, for every $p \in Y$. We compute
\begin{equation*}
(\EE_{H(Y,\EE)})_p ~=~ \{ [a]_{\sim_p} : a \in H(Y,\EE),~p(a) = 1 \}
\end{equation*}
and $\lambda_p$ is given by $\lambda_p([a]_{\sim_p}) = \mathrm{germ}_p(a)$. The statement now follows from Lemma \ref{lmm:stalks-vs-equivalence-rel}.
\end{proof}

\section{Separation properties}
\label{sec:separation}

Recall that $P_{\{ 1_a, 1_b \}}$ denotes the primitive skew lattice with top elements $1_a$ and $1_b$. From the results of \cite{kcv-priestley} it easily follows that for every two elements $a,b \in A$ in the same $\mathcal{D}$-class, there is a morphism of skew lattices
\begin{equation*}
q : A \to P_{\{1_a,1_b\}}
\end{equation*}
such that $q(a)=1_a$ and $q(b)=1_b$. We say that $q$ \emph{separates $a$ and $b$}. However, this morphism $q$ is in general not a morphism of noncommutative frames. In fact:

\begin{proposition} \label{prop:separation}
Let $A$ be a noncommutative frame such that $A/\mathcal{D}$ is spatial. Then the following are equivalent:
\begin{enumerate}
\item for all $a,b \in A$ in the same $\mathcal{D}$-class such that $a \neq b$, there is a morphism of noncommutative frames
\begin{equation*}
q : A \to P_{\{1_a,1_b\}}
\end{equation*}
such that $q(a) = 1_a$ and $q(b) = 1_b$;
\item the natural map $\sigma : A \longrightarrow H(G(A)),~ a \mapsto (U_a,s_a)$ is injective;
\item there is an injective morphism of noncommutative frames $A \to A'$ for some spatial noncommutative frame $A'$;
\item there is an injective morphism of noncommutative frames $A \to \PP(R,S)$ for some sets $R$ and $S$.
\end{enumerate}
\end{proposition}
\begin{proof}
\underline{$(1)\Rightarrow(2)$}.~ Suppose that $\sigma(a) = \sigma(b)$. If $a$ and $b$ have different $\mathcal{D}$-class, we can use that $A/\mathcal{D}$ is spatial to construct a morphism $p : A \longrightarrow \mathbf{2}$ such that $p(a) \neq p(b)$. But $\mathbf{2}$ is spatial as noncommutative frame, so there is a factorization $p = \xi \circ \sigma$ for some $\xi : H(G(A)) \to \mathbf{2}$. This leads to a contradiction. If $a$ and $b$ have the same $\mathcal{D}$-class, then since $(1)$ holds, we can construct a morphism of noncommutative frames $q : A \to P_{\{1_a,1_b\}}$ such that $q(a) \neq q(b)$. Again there is a factorization $q = \xi \circ \sigma$, a contradiction.

\underline{$(2)\Rightarrow(3)$}.~ This is trivial, because $H(G(A))$ is spatial.

\underline{$(3)\Rightarrow(4)$}.~ Suppose $A' = H(Y,\EE)$, and let $\pi : E \to Y_f$ be the \'etale space corresponding to $\EE$. Then the inclusion
\begin{equation*}
A' \subseteq \PP(Y,E)
\end{equation*}
is a morphism of noncommutative frames. But then we can also embed $A$ in $\PP(Y,E)$.

\underline{$(4)\Rightarrow(1)$}.~ Suppose that there is an injection $i : A \to \PP(R,S)$. Take $a,b \in A$ in the same $\mathcal{D}$-class, with $a \neq b$. Then $i(a)$ and $i(b)$ are function $U \to S$ for some subset $U \subseteq R$. Moreover, we can find $p \in U$ such that $i(a)$ and $i(b)$ have different values in $p$. We now consider the map
\begin{equation*}
\xi : \PP(R,S) \to \PP(\{p\},S),\quad \xi(f) = f|_{\{p\}}.
\end{equation*}
Then composition gives a morphism $q': A \to P_S$ from $A$ to the primitive skew lattice with $S$ as set of top elements, such that $q'(a)$ and $q'(b)$ are different top elements. Now take a quotient $\PP_S \to \PP_{\{1_a,1_b\}}$ sending $q'(a)$ to $1_a$ and $q'(b)$ to $1_b$.
\end{proof}

\begin{example}
Consider the noncommutative frame $A'$ from Subsection \ref{ssec:not-spatial}. We showed in Subsection \ref{ssec:proof-not-spatial} that
\begin{equation*}
H(G(A')) = \PP(\mathbb{R}).
\end{equation*}
The natural map $A' \longrightarrow H(G(A'))$ sends $(U,f)$ to $U$, and as a result this map is surjective but not injective. So none of the equivalent statements from \mbox{Proposition \ref{prop:separation}} hold in this case, for example there is no embedding of noncommutative frames of $A'\subseteq \PP(R,S)$ for sets $R$ and $S$.

For the noncommutative frame $A''$ from Subsection \ref{ssec:not-spatial} there is an embedding $A'' \subseteq \PP(R,S)$. So here the equivalent conditions from \mbox{Proposition \ref{prop:separation}} are satisfied.
\end{example}

\section{Sheaves on the dissolution locale}
\label{sec:dissolution}

Theorem \ref{thm:duality-spatial} shows that we can identify a spatial noncommutative frame with its corresponding pair $(Y,\EE)$, where $Y$ is a sober space and $\EE$ is a sheaf on $Y_f$. In the remaining part of the paper, we discuss a duality result for general left-handed noncommutative frames. More precisely, we want to show that noncommutative frames with commutative shadow $\OO(Y)$, for $Y$ a locale, correspond to sheaves on the dissolution locale $Y_d$. In this section we recall the definition of dissolution locale and give a description of the sheaves on it. Note that in the literature, e.g.\ in \cite[C1.1]{johnstone-elephant-2} and \cite{frame-of-nuclei-spatial}, results on the dissolution locale are often phrased in terms of its associated frame, called the frame of nuclei or assembly.

\subsection{The dissolution locale}
\label{ssec:dissolution-pullback}

Let $Y$ be a locale and let $L = \OO(Y)$ be the associated frame. A \emph{nucleus} on $L$ is a function $\nu : L \to L$ satisfying
\begin{enumerate}
\item[(N1)] $a \leq \nu(a)$;
\item[(N2)] $\nu(a \meet b) = \nu(a) \meet \nu(b)$;
\item[(N3)] $\nu(\nu(a)) = \nu(a)$
\end{enumerate}
for $a, b \in L$. Each nucleus defines a surjective morphism of frames 
\begin{equation*}
\nu : L \longrightarrow \nu(L),\quad a \mapsto \nu(a)
\end{equation*}
and up to isomorphism every surjective frame morphism is of this form. In the category of frames, the regular epimorphisms are precisely the surjective morphisms. In this way, we see that nuclei on $L$ correspond bijectively to \emph{sublocales} of Y, i.e.\ with isomorphism classes of regular monomorphisms $Y' \longrightarrow Y$. The nuclei on $L$ are partially ordered via
\begin{equation*}
\nu \leq \nu' ~\Leftrightarrow~ \nu(a) \leq \nu'(a),~\forall a \in L.
\end{equation*}
With this partial order, the set of nuclei $N(L)$ is a frame, see \cite[C1.1]{johnstone-elephant-2}, called the \emph{frame of nuclei} or the \emph{assembly}. The locale dual to $N(L)$ is denoted by $Y_d$ and is called the \emph{dissolution locale} of $Y$. There is an embedding of frames
\begin{equation*}
L \longrightarrow N(L),\quad a \mapsto \nu_a
\end{equation*}
with the nucleus $\nu_a$ defined by $\nu_a(b) = a \join b$.

\subsection{The dissolution locale as a pullback}

For every distributive lattice $L$ there is a boolean algebra $L_b \supseteq L$ with the property that any lattice homomorphism $L \longrightarrow B$ to a boolean algebra $B$ extends to a lattice homomorphism $L_b \longrightarrow B$ in a unique way. We call $L_b$ the \emph{boolean envelope} of $L$, following the terminology of Banaschewski in \cite{banaschewski}. The frame of nuclei $N(L)$ plays a similar role for frames: if $\phi: L \to B$ is a morphism of frames to a frame $B$, such that $\phi(a)$ is complemented for all $a \in L$, then there is a unique frame morphism $N(L) \longrightarrow B$ extending $\phi$, see \cite[C1.1]{johnstone-elephant-2}. With these universal properties in mind, it is easy to show that there is a pushout diagram
\begin{equation*}
\begin{tikzcd}
\mathrm{Idl}(L_b) \ar[r] & N(L) \\
\mathrm{Idl}(L) \ar[u] \ar[r] & L \ar[u]
\end{tikzcd}
\end{equation*}
where $L_b$ is the boolean envelope of $L$ (as a distributive lattice), and $\mathrm{Idl}(L)$ and $\mathrm{Idl}(L_b)$ are the ideal completions of $L$ resp.\ $L_b$. This pushout diagram already appeared in \cite{klinke-slides}.

Dually, with $Y$ the locale corresponding to $L$, we can write the dissolution locale $Y_d$ as a pullback. Let 
\begin{equation*}
X = \mathrm{Spec}(L) 
\end{equation*}
be the prime spectrum of $L$. The patch topology on $X$, as introduced by Hochster in \cite{hochster}, is the topology generated by the compact open subsets in $X$ and their complements. We will write $X_p$ for $X$ with the patch topology; this is a compact Hausdorff space such that the clopen subsets are a basis for the topology. Then $X$ and $X_p$ are the locales corresponding to $\mathrm{Idl}(L)$ resp.\ $\mathrm{Idl}(L_b)$. So the following is a pullback diagram in the category of locales:
\begin{equation} \label{eq:pullback-dissolution}
\begin{tikzcd}
Y_d \ar[r,"{j}"] \ar[d,"{\delta}"'] & X_\patch \ar[d,"{\pi}"] \\
Y \ar[r,"{i}"'] & X
\end{tikzcd}.
\end{equation}

\subsection{Sheaves on the dissolution locale }
\label{ssec:sheaves-on-dissolution}
 The pullback diagram (\ref{eq:pullback-dissolution}) gives rise to a commutative diagram
\begin{equation} \label{eq:pullback-pushforwards}
\begin{tikzcd}
\sh(Y_d) \ar[r,"{j_*}"] \ar[d,"{\delta_*}"'] & \sh(X_\patch) \ar[d,"{\pi_*}"] \\
\sh(Y) \ar[r,"{i_*}"'] & \sh(X)
\end{tikzcd}
\end{equation}
between the associated sheaf toposes. Note that $i_*$ and $j_*$ are the pushforward maps corresponding to the sublocales $Y \subseteq X$ resp.\ $Y_d \subseteq X_\patch$. So both $i_*$ and $j_*$ are fully faithful. This means that, up to equivalence, we can identify $\sh(Y)$ and $\sh(Y_d)$ with their essential images in $\sh(X)$ resp.\ $\sh(X_p)$. We then say that a sheaf on $X$ \emph{is a sheaf on $Y$} if it is in the essential image of $i_*$. Similarly, we say that a sheaf on $X_p$ \emph{is a sheaf on $Y_d$} if it is in the essential image of $j_*$.

We want to give a concrete description of the sheaves on $Y_d$ in terms of sheaves on $Y$, $X$ and $X_p$. A first step is the following:

\begin{lemma} \label{lmm:sheaves-on-dissolution}
Let $\GG$ be a sheaf on $X_p$. The following are equivalent:
\begin{enumerate}
\item $\GG$ is a sheaf on $Y_d$ (in the sense explained above);
\item for every compact open $U \subseteq X$ and compact opens $U_i \subseteq U$, $i \in I$, such that $i^{-1}(U) = \bigcup_{i \in I} i^{-1}(U_i)$, we have the following ``stable'' sheaf condition: for every clopen $Z \subseteq X_p$ and family of sections
\begin{equation*}
s_i \in \GG(U_i \cap Z),~ i \in I
\end{equation*}
agreeing on intersections, there is a unique section $s \in \GG(U \cap Z)$ such that $s|_{U_i \cap Z} = s_i$ for all $i \in I$. 
\end{enumerate}
\end{lemma}
\begin{proof}
\underline{$(1)\Rightarrow(2)$}. Suppose that $\GG \cong j_*\FF$ for $\FF$ in $\sh(Y_d)$. Then this follows from
\begin{gather*}
j^{-1}\left(\bigcup_{i \in I} U_i \cap Z  \right) = j^{-1}\left(\bigcup_{i \in I} U_i\right) \cap j^{-1}(Z) \\ = j^{-1}(U) \cap j^{-1}(Z) = j^{-1}(U \cap Z).
\end{gather*}
\underline{$(2)\Rightarrow(1)$}. Let $\BB_p$ be the set of clopen subsets of $X_p$. We interpret $\BB_p$ as a category with a unique arrow $Z \longrightarrow Z'$ whenever $Z \subseteq Z'$. Every sheaf $\GG$ on $X_p$ restricts to a presheaf on $\BB_p$. We now consider two Grothendieck topologies on $\BB_p$:
\begin{itemize}
\item the \emph{coherent topology} $J_\coh$ with as covering sieves on $Z \in \BB_p$ the sieves containing elements $Z_1,\dots,Z_k$ such that $Z = Z_1 \cup \dots \cup Z_k$;
\item the \emph{lifted Grothendieck topology} $J_L$ generated by the covering families $\{Z_i \to Z \}_{i \in I}$ such that there exists a compact open $U \subseteq X$ and compact opens $U_i \subseteq U$ satisfying
\begin{equation*}
Z_i = U_i \cap Z,~\forall i \in I \quad\text{and}\quad \bigcup_{i \in I} i^{-1}(U_i) = i^{-1}(U).
\end{equation*}
\end{itemize}
Note that $\sh(\BB_p,J_\coh) \simeq \sh(X_p)$, since $\BB_p$ is a basis for the topology with each $Z$ in $\BB_p$ compact. We have to prove that 
\begin{equation} \label{eq:intersection-of-toposes}
\sh(Y_d) ~\simeq~ \sh(\BB_p,J_\coh) ~\cap~ \sh(\BB_p,J_L)
\end{equation}
or in other words, $\sh(Y_d) \simeq \sh(\BB_p,J_\coh \join J_L)$, where $J_\coh \join J_L$ is the smallest Grothendieck topology containing $J_\coh$ and $J_L$. To see that this proves the statement, note that if $\GG$ satisfies $(2)$ then $\GG$ is a sheaf with respect to the covering families generating $J_L$. Since these covering families are stable under pullback, they are a coverage in the sense of Johnstone \cite[A2.1, Definition 2.1.9]{johnstone-elephant-1}, which means that $\GG$ is a sheaf for the Grothendieck topology $J_L$ generated by them. Because $\GG$ is a sheaf on $X_p$, it is a sheaf for $J_\coh$ as well, so $\GG$ is a sheaf on $Y_d$.

We now show that (\ref{eq:intersection-of-toposes}) holds. We already showed that if $\GG$ is a sheaf on $Y_d$, then it restricts to a sheaf for both $J_\coh$ and $J_L$. Conversely, suppose that $\GG$ restricts to a sheaf for both $J_\coh$ and $J_L$. Then $\pi_*\GG$ is a sheaf on $Y$, so there is a factorization
\begin{equation*}
\begin{tikzcd}
\sh(\BB_p,J_\coh \join J_L) \ar[r] \ar[d] & \sh(X_p) \ar[d,"{\pi_*}"] \\
\sh(Y) \ar[r,"{i_*}"] & \sh(X)
\end{tikzcd}.
\end{equation*}
Since \ref{eq:pullback-dissolution} is a pullback diagram, $Y_d$ is the biggest sublocale of $X_p$ such that $\pi \circ j$ factors through $i$. The topos $\sh(\BB_p,J_\mathrm{coh}\vee J_L)$ defines a sublocale with the same universal property. So (\ref{eq:intersection-of-toposes}) holds.
\end{proof}

The sheaves on $Y_d$ with nonempty set of global sections can be described with the following criterion. For the abuse of language ``is a sheaf on [\dots]'' we refer to the beginning of this subsection.

\begin{theorem} \label{thm:sheaves-on-dissolution}
Let $\GG$ be a sheaf on $X_p$ with $\GG(X_p) \neq \varnothing$. Then the following are equivalent:
\begin{enumerate}
\item $\GG$ is a sheaf on $Y_d$;
\item $\pi_*\GG$ is a sheaf on $Y$.
\end{enumerate}
\end{theorem}
\begin{proof}
The direction $(1) \Rightarrow (2)$ is trivial, we prove the other direction. Suppose $\pi_*\GG$ is a sheaf on $Y$. We show that property $(2)$ of Lemma \ref{lmm:sheaves-on-dissolution} holds. So take $U$, $Z$, $(U_i)_{i \in I}$, $(s_i)_{i \in I}$ like in $(2)$ from Lemma \ref{lmm:sheaves-on-dissolution}. We can write $Z$ as a finite union of sets of the form $U'-V'$ with $V' \subseteq U' \subseteq X$ compact open. Without loss of generality, we can assume that $Z$ itself is of this form, say $Z = U' - V'$. We can further assume that $U' = U$, so $Z_i = U_i - V$. Take a global section $t \in \GG(X_p)$. We write $s_i'$ for the unique section on $U_i \cup V$ restricting to $s_i$ on $Z_i$ and to $t|_V$ on $V$. Since $\pi_*\GG$ is a sheaf on $Y$, the sections $s_i'$ glue to a unique section $s' \in \GG(U)$. Then the section $s = s'|_Z$ is the unique gluing of the family $(s_i)_{i \in I}$.
\end{proof}

\section{Duality for general noncommutative frames}
\label{sec:duality-general-ncframe}

In this section we prove the duality between noncommutative frames and sheaves on dissolution locales. This is a variation on the noncommutative Priestley duality developed in \cite{kcv-priestley}. So we start by recalling Priestley duality, noncommutative Priestley duality, and the associated terminology.

\subsection{Priestley duality} We discuss some classical results and some results and terminology from \cite{kcv-priestley}. 

For $D$, $D'$ distributive lattices, a map $f: D \longrightarrow D'$ is called \emph{proper} if 
\begin{equation*}
\forall y \in D',~\exists x \in D,~y \leq f(x). 
\end{equation*}
Now $\mathbf{DL}_0$ denotes the category of distributive lattices with a least element $0$, with as morphisms the proper maps preserving $\meet$, $\join$ and $0$. Further $\mathbf{DL}_{01}$ denotes the category of distributive lattices with a least element $0$ and a greatest element $1$, with as morphisms the maps preserving $\meet$, $\join$, $0$ and $1$. The latter maps are automatically proper, so $\mathbf{DL}_{01}$ is a full subcategory of $\mathbf{DL}_0$. The lattices in $\mathbf{DL}_{01}$ are called the \emph{bounded distributive lattices}.

Recall that $\mathbf{2}$ denotes the bounded distributive lattice with as only elements $0$ and $1$. Take $D$ in $\mathbf{DL}_0$. A \emph{prime filter on $D$} is a proper map $p: D \longrightarrow \mathbf{2}$ preserving $0$, $\meet$ and $\join$. By looking at the preimages of $1 \in \mathbf{2}$ we can alternatively describe a prime filter as a nonempty upwards closed subset $F \subset D$ that is closed under $\meet$, does not contain $0$, and satisfies $a \vee b \in F ~\Rightarrow~ a \in F \text{ or } b \in F$. For every $a \in D$, we can define
\[
\widehat{a} = \{ F ~\text{prime filter} : a \in F \}
\]
The sets $\widehat{a}$ are the basis for a topology on the set of prime filters on $D$.

\begin{definition}
Let $D$ be a distributive lattice with $0$. Then the \emph{prime spectrum} $\Spec(D)$ of $D$ is the set of prime filters, equipped with the topology generated by the subsets $\widehat{a} ~\subseteq~ \Spec(D)$ for $a \in D$.
\end{definition}

It is easy to see that the compact open subsets of $\Spec(D)$ are in bijective correspondence with the elements of $D$. In particular, $\Spec(D)$ is compact if and only if $D$ is bounded.

Recall from \cite{hochster} that a topological space $X$ is called a \emph{spectral space} if the following conditions are satisfied:
\begin{itemize}
\item $X$ is compact;
\item $X$ is sober;
\item the compact open subsets of $X$ are closed under finite intersections and form a basis for the topology.
\end{itemize}
A \emph{locally spectral space} is a space that can be covered by open subsets that are spectral spaces. A continuous map of (locally) spectral spaces is called \emph{spectral} if the inverse images of compact open sets are again compact open. The category of spectral spaces (and spectral maps between them) is called $\mathbf{Spec}$. The full subcategory of locally spectral spaces is called $\mathbf{LSpec}$. The following is classical:

\begin{proposition}
The assignment $D \mapsto \Spec(D)$ defines an equivalence of categories between $\mathbf{DL}_0$ and $\mathbf{LSpec}$, and between $\mathbf{DL}_{01}$ and $\mathbf{Spec}$.
\end{proposition}

For the specialization order on a (locally) spectral space $X$, we will use the convention from \cite{frame-of-nuclei-spatial}, i.e.\ 
\begin{equation*}
x \leq y \quad\Leftrightarrow\quad x \in \overline{\{y\}}.
\end{equation*}
So for two filters $F$ and $F'$ we get $F \leq F'$ if and only if $F \subseteq F'$. This is the opposite convention as in \cite{kcv-priestley}.
For a (locally) spectral space $X$, the \emph{patch topology}, as introduced by Hochster in \cite{hochster}, is the topology generated by the compact open subsets of $X$ and their complements. We write $X_\patch$ for the set $X$ equipped with the patch topology. Note that $X_p$ is Hausdorff; in particular, the specialization order on $X_p$ is trivial.

A \emph{Priestley space} is a triple $(X,\tau,\leq)$ with $X$ a set, $\tau$ a topology on $X$, and $\leq$ a partial order on $X$, such that $(X,\tau)$ is compact and moreover
\begin{align}
\begin{split}
\forall x,y \in X \text{ such that }x \not\leq y,~ \text{there are disjoint} \\ \text{clopen subsets }U \in \tau^\uparrow,~V \in \tau^\downarrow,~ x \in U,~y \in V.
\end{split}
\end{align}
We write $\mathbf{PS}$ for the category of Priestley spaces and continuous monotonous maps between them. If $D$ is a bounded distributive lattice, then it is straightforward to check that
\begin{equation*}
(\Spec(D),\mathsf{patch},\leq)
\end{equation*}
is a Priestley space (with $\mathsf{patch}$ the patch topology and $\leq$ the specialization order on $\Spec(D)$), and that conversely every Priestley space is of this form. This induces a equivalence of categories $\mathbf{DL}_{01}^\op \simeq \mathbf{PS}$, called \emph{Priestley duality}. It first appeared in the work of Priestley \cite{priestley-duality}, where it was shown directly, without using the existing duality between bounded distributive lattices and spectral spaces. For more details on the two dualities and their interaction, we refer to the paper \cite{cornish-priestley} by Cornish.

In \cite{kcv-priestley} it was shown that the same functor 
\begin{equation*}
D ~\mapsto~ (\Spec(D),\mathsf{patch},\leq)
\end{equation*}
defines a duality between $\mathbf{DL}_0$ and the full subcategory $\mathbf{LPS} \subseteq \mathbf{PS}$ of \emph{local Priestley spaces}, see \cite[Subsection 3.2]{kcv-priestley} for the definition and the proof. However, the distributive lattices $D$ appearing in this paper will often be frames, and since frames always have a top element, the associated local Priestley space $(\Spec(D),\mathsf{patch},\leq)$ is a Priestley space. 

We can say even more. Recall from \cite[Section 3]{frame-of-nuclei-spatial} that a Priestley space $(X,\tau,\leq)$ is called an \emph{Esakia space} if whenever $U$ is $\tau$-clopen, the downset $U\!\!\downarrow$ is $\tau$-clopen as well. Further, an Esakia space is called \emph{extremally order-disconnected} if the closure of each $\tau$-open upset is $\tau$-clopen. Then for $D$ a bounded distributive lattice, it turns out that the associated Priestley space is an Esakia space if and only if $D$ is a Heyting algebra, and that $D$ is an extremally order-disconnected Esakia space if and only if $D$ is a frame, see \cite[Section 3]{frame-of-nuclei-spatial} and originally Pultr and Sichler \cite{pultr-sichler}.

Let $D$ and $D'$ be frames. Consider a continuous monotonous map
\begin{equation*}
\psi : (\Spec(D),\mathsf{patch},\leq) \longrightarrow (\Spec(D'),\mathsf{patch},\leq)
\end{equation*}
between the associated extremally ordered-disconnected Esakia spaces. Then $\psi$ corresponds to a morphism $g : D \longrightarrow D'$ in $\mathbf{DL}_{01}$. It is natural to ask what conditions on $\psi$ are necessary and sufficient such that $g$ is a frame morphism. This is another question solved in \cite{pultr-sichler}. We will not need their result in this paper, since we will formulate everything in terms of the underlying locales.

\subsection{Noncommutative Priestley duality}
\label{ssec:nc-frames-duality}

We recall the following from \cite{kcv-priestley}. Let $\sh(\mathsf{LPS})$ be the category of pairs $(X,\FF)$ where $X$ is a local Priestley space and $\FF$ is a sheaf on $X$ with global support; morphisms $(X,\FF) \to (X',\FF')$ are given by a pair $(f,\lambda)$, with $f : X \to X'$ continuous monotonous and $\lambda : \FF' \to f_*\FF$ a sheaf morphism. Further, let $\mathbf{SDL}$ be the category of left-handed strongly distributive skew lattices with $0$, with as morphisms the skew lattice homomorphisms $\phi : S \to S'$ preserving $0$, that are \emph{proper} in the sense that $\forall y \in S',~\exists x \in S,~ y \leq \phi(x)$.

Consider the functor
\begin{gather*}
(-)^\ast : \sh(\mathsf{LPS})^\op \longrightarrow \mathbf{SDL}
\end{gather*}
with $A = (X,\FF)^\ast$ the skew lattice defined as follows. The elements are pairs $(U,s)$ with $U$ an upwards closed compact open set and $s \in \FF(U)$ a section. Meet and join are defined by
\begin{align*}
&(U,s) \meet (V,t) = (U \cap V, s|_{U \cap V}) \\
&(U,s) \join (V,t) = (U \cup V, s|_{U-V} \cup t|_V)
\end{align*}
where $s|_{U-V} \cup t|_V$ is the unique section on $U \cup V$ that restricts to $s|_{U-V}$ on $U-V$ and to $t|_V$ on $V$.

\begin{theorem}[{\cite[Theorem 3.7]{kcv-priestley}}] \label{thm:nc-priestley}
The functor $(-)^\ast$ induces a dual equivalence between the category $\mathbf{SDL}$ is and the category $\sh(\mathsf{LPS})$.
\end{theorem}

For $A = (X,\FF)^\ast$, the commutative shadow $A/\mathcal{D}$ is the unique distributive lattice such that $\Spec(A/\mathcal{D}) = X$, see the previous subsection. Note that $A/\mathcal{D}$ is a bounded distributive lattice if and only if $X$ is a Priestley space. So let $\sh(\mathsf{PS}) \subset \sh(\mathsf{LPS})$ be the full subcategory consisting of the pairs $(X,\FF)$ such that $X$ is a Priestley space, and let $\mathbf{SDL}_\mathrm{01} \subset \mathbf{SDL}$ be the full subcategory consisting of the skew lattices $A$ in $\mathbf{SDL}$ such that $A/\mathcal{D}$ is bounded. Then we find:

\begin{corollary}
The functor $(-)^\ast$ induces a dual equivalence between the category $\mathbf{SDL}_\mathrm{01}$ and the category $\sh(\mathsf{PS})$.
\end{corollary}

The inverse functor to $(-)^\ast$ is written as $(-)_\ast$. For the explicit description of $(-)_\ast$ we refer to \cite{kcv-priestley}. In hindsight, we can of course define $A_\ast$ as the unique pair $(X,\FF)$, up to isomorphism, such that $A \cong (X,\FF)^\ast$. This is sometimes already enough to compute $A_\ast$ in practice, as we will demonstrate in the following example.

\begin{example}
Let $A = \PP(R,S)$ be the noncommutative frame of partial functions from $R$ to $S$, as in Example \ref{ex:partial-func}. We write $X = \Spec(A/\mathcal{D})$. Then $A_\ast = (X',\FF)$ where $X' = (X,\mathrm{patch},\leq)$ is the Priestley space associated to $X$, and $\FF$ is a sheaf on $X'$ (or in other words a sheaf on $X_p$ where $X_p$ is $X$ with the patch topology).

We know that $A/\mathcal{D}\cong\PP(R)$. The elements of $X=\Spec(A/\mathcal{D})$ are then the prime filters in $\PP(R)$. Since $\PP(R)$ is a boolean algebra, the prime filters are exactly the ultrafilters. So the elements of $X$ are the ultrafilters of $\PP(R)$, and the topology on $X$ is generated by the subsets
\[
\hat{U} = \{ F \text{ ultrafilter in }\PP(R) : U \in F \}
\]
for $U \in \PP(R)$. In other words, $X$ is the Stone--\v{C}ech compactification of $R$ (where $R$ has the discrete topology).

Note that $X$ is Hausdorff, and as a result $X_p = X$. So $\FF$ is just a sheaf on $X$. It necessarily satisfies
\begin{equation*}
\FF(\hat{U}) = \{ f: U \to S \}
\end{equation*}
(the set of all functions from $U$ to $S$), for $U \in \PP(R)$. Since the sets $\hat{U}$ are a basis for the topology, this uniquely determines $\FF$.
\end{example}

\subsection{Duality for noncommutative frames}

We want to show that left-handed noncommutative frames correspond to pairs $(Y,\FF)$, where $Y$ is a locale and $\FF$ is a sheaf on the dissolution locale $Y_d$ such that $\FF(Y_d) \neq \varnothing$.

So let $A$ be a left-handed strongly distributive skew lattice with $0$ such that $A/\mathcal{D}$ is a frame. Define $Y$ to be the locale with $\OO(Y) = A/\mathcal{D}$. Then there is a pullback diagram
\begin{equation} \label{eq:pullback-dissolution-2}
\begin{tikzcd}
Y_d \ar[r,"{j}"] \ar[d,"{\delta}"'] & X_\patch \ar[d,"{\pi}"] \\
Y \ar[r,"{i}"'] & X
\end{tikzcd}.
\end{equation}
where $X = \Spec(A/\mathcal{D})$ and $X_p$ is $X$ with the patch topology, see Section \ref{sec:dissolution}. Further, $A_\ast = (X,\GG)$ with $\GG$ some sheaf on $X_p$.

What is a necessary and sufficient condition on $\GG$ such that $A$ is a noncommutative frame? We will use the following criterion:

\begin{theorem}[{\cite[Theorem 5.1]{completeness-issues}}] \label{thm:completeness-issues}
Let $A$ be a strongly distributive skew lattice with $0$ such that $A/\mathcal{D}$ is a frame. Then $A$ is a noncommutative frame if and only if $A$ is join complete.
\end{theorem}

This leads to the following criterion. For the abuse of language ``is a sheaf on [\dots]'' we refer to the beginning of Subsection \ref{ssec:sheaves-on-dissolution}.

\begin{proposition} \label{prop:objects}
With the above notations, the following are equivalent:
\begin{enumerate}
\item $A$ is a noncommutative frame;
\item $\pi_*\GG$ is a sheaf on $Y$;
\item $\GG$ is a sheaf on $Y_d$.
\end{enumerate}
\end{proposition}
\begin{proof}
We identify $A$ with $(X,\GG)^\ast$. The elements of $A$ are then pairs $(U,s)$ where $U$ is a compact open subset of $X$ and $s \in \GG(U)$ is a section.

\underline{$(1)\Rightarrow (2)$}.~ Note that the compact open subsets of $X$ form a frame, isomorphic to the frame of open subsets of $Y$. Let $U \subseteq X$ be an open set, and let $\nu(U)$ be the smallest compact open set containing $U$. To show that $\pi_*\GG$ is a sheaf on $Y$, it is enough to prove that the restriction map
\begin{equation} \label{eq:restriction-map-nu}
\rho : \GG(\nu(U)) \to \GG(U)
\end{equation}
is bijective. We can identify $\GG(\nu(U))$ with the $\mathcal{D}$-class in $A$ determined by the compact open subset $\nu(U) \subseteq X$. On the other hand, the elements of $\GG(U)$ correspond to the commuting subsets $\{ (V,s_V) : V \subseteq U \text{ compact open} \} \subseteq A$. The inverse to the map $\rho$ in (\ref{eq:restriction-map-nu}) is then given by taking joins of commuting families.

\underline{$(2)\Rightarrow(1)$}.~ By Theorem \ref{thm:completeness-issues}, it is enough to prove that $A$ is join complete. Consider $(V_i,s_i) \in A$ a commuting family indexed by $i \in I$. Let $V = \bigvee_{i \in I} V_i$ be the join in $A/\mathcal{D}$, or in the notation above $V = \nu\left(\bigcup_{i \in I} V_i \right)$. By \cite[Proposition 4.1]{completeness-issues}, the family $(V_i,s_i)$ has a join if and only if there is a unique element $(V,s) \in A$ such that $(V_i,s_i) \leq (V,s)$ for all $i \in I$, or in other words $s|_{V_i} = s_i$ for all $i \in I$. This is precisely the sheaf condition.

\underline{$(2)\Leftrightarrow(3)$}.~ This follows from Theorem \ref{thm:sheaves-on-dissolution}, provided that $\GG(X_p) \neq \varnothing$. But this follows because the elements of $\GG(X_p)$ are in bijective correspondence with the elements of $A$ that are in the top $\mathcal{D}$-class.
\end{proof}

This shows that left-handed noncommutative frames with commutative shadow $L = \OO(Y)$ are in bijective correspondence with sheaves $\FF$ on the dissolution locale $Y_d$ such that $\FF(Y_d) \neq \varnothing$. What about the morphisms?

\begin{proposition} \label{prop:morphisms}
Take noncommutative frames $A$ and $A'$ and a morphism $\phi : A \to A'$ in $\mathbf{SDL}$, i.e.\ $\phi$ is a proper map preserving $\meet$, $\join$ and $0$. If the induced morphism on commutative shadows $\phi/\mathcal{D} : A/\mathcal{D} \to A'/\mathcal{D}$ is a morphism of frames, then $\phi$ is a morphism of noncommutative frames.
\end{proposition}
\begin{proof}
Let $(a_i)_{i \in I}$ be a commuting family in $A$. Note that 
\begin{equation} \label{eq:inequality}
\bigvee_{i \in I} \phi(a_i) \leq \phi\left( \bigvee_{i \in I} a_i \right).
\end{equation}
Recall that $[x]$ denotes the $\mathcal{D}$-class of an element $x$. Then we compute:
\begin{equation*}
\left[\bigvee_{i \in I} \phi(a_i) \right] = \bigvee_{i \in I}(\phi/\mathcal{D})([a_i])
= (\phi/\mathcal{D})\left(\left[\bigvee_{i \in I} a_i\right]\right) = \left[\phi\left( \bigvee_{i \in I} a_i \right)\right]
\end{equation*}
where in the second equality we use that $(\phi/\mathcal{D})$ is a frame morphism. Since (\ref{eq:inequality}) is an inequality between two elements in the same $\mathcal{D}$-class, it must be an equality. This shows that $\phi$ is a morphism of noncommutative frames.
\end{proof}

We can now formulate Proposition \ref{prop:objects} in a categorical way. We define the \emph{category of sheaves on dissolution locales} $\sh(\mathsf{Loc}_d)$ as the category with
\begin{itemize}
\item as objects the pairs $(Y,\FF)$ with $Y$ a locale and $\FF$ a sheaf on the dissolution locale $Y_d$ such that $\FF(Y_d) \neq \varnothing$;
\item as morphisms the pairs $(f,\lambda) : (Y,\FF) \to (Y',\FF')$ with $f : Y \to Y'$ a locale morphism and
\begin{equation*}
\lambda : \FF' \to g_*\FF
\end{equation*}
a sheaf morphism, where $g : Y_d \to Y'_d$ is the morphism induced by $f$.
\end{itemize}

Then by combining Proposition \ref{prop:objects} and Proposition \ref{prop:morphisms} we find:

\begin{theorem}[Duality for general left-handed noncommutative frames] \label{thm:duality}
The duality of Theorem \ref{thm:nc-priestley} (originally \cite[Theorem 3.7]{kcv-priestley}) restricts to a dual equivalence between the category $\mathbf{LNFrm}$ of left-handed noncommutative frames and the category $\sh(\mathsf{Loc}_d)$ of sheaves on dissolution locales.
\end{theorem}

Take a pair $(Y,\FF)$ in $\sh(\mathsf{Loc}_d)$. Consider the pullback diagram
\begin{equation} \label{eq:pullback-dissolution-3}
\begin{tikzcd}
Y_d \ar[r,"{j}"] \ar[d,"{\delta}"'] & X_\patch \ar[d,"{\pi}"] \\
Y \ar[r,"{i}"'] & X
\end{tikzcd}.
\end{equation}
Then the left-handed noncommutative frame corresponding to $(Y,\FF)$ is given by
\begin{equation} \label{eq:duality-1}
A(Y,\FF) = \{ (U,s) : U \in \OO(Y),~ s \in \FF(\delta^{-1}(U)) \}
\end{equation}
with meet and join defined by
\begin{equation} \label{eq:duality-2}
\begin{split}
&(U,s) \meet (V,t) = (U \cap V, s|_{U \cap V}) \\
&(U,s) \join (V,t) = (U \cup V, s|_{U-V} \cup t)
\end{split}
\end{equation}
where $s|_{U-V} \cup t$ is the unique section on $U \cup V$ restricting to $s|_{U-V}$ on $U-V$ and to $t$ on $V$.

Conversely, this means that the pair $(Y,\FF)$ satisfies
\begin{equation*}
\OO(Y) = A(Y,\FF)/\mathcal{D}
\end{equation*}
and
\begin{equation*}
\FF(\delta^{-1}(U)) = \{ a \in A(Y,\FF) : [a] = U \}.
\end{equation*}

\subsection{Characterization of spatial noncommutative frames}

We end the paper with a characterization of spatial noncommutative frames in terms of the duality $\mathbf{LNFrm}^\op \simeq \sh(\mathsf{Loc}_d)$.
This is based on the following relation between $Y_f$ and the dissolution locale $Y_d$:

\begin{proposition}[{\cite[Proposition 5.4]{frame-of-nuclei-spatial}}]
Let $Y$ be a sober topological space. Then $Y_f = \pt(Y_d)$. So $Y_f$ corresponds to the largest spatial sublocale of $Y_d$. More generally, if $Y$ is a locale, then $\pt(Y)_f = \pt(Y_d)$.
\end{proposition}

We have then a commutative diagram
\begin{equation} \label{eq:k-delta}
\begin{tikzcd}
Y_f \ar[r,"{k}"] \ar[rd,equal] & Y_d \ar[d,"{\delta}"]\\
& Y 
\end{tikzcd}
\end{equation}
where $k : Y_f \to Y_d$ is the inclusion of $Y_f$ as a sublocale, and $\delta$ is the morphism from (\ref{eq:pullback-dissolution-3}).

\begin{proposition} \label{prop:comparison} Let $A$ be a left-handed noncommutative frame. Let $(Y,\FF)$ be the dual of $A$ in $\sh(\mathsf{Loc}_d)$, in particular $\OO(Y) = A/\mathcal{D}$. Then the following are equivalent:
\begin{enumerate}
\item $A$ is spatial;
\item $Y$ is spatial and $\FF \cong k_*\EE$ for some sheaf $\EE$ on $Y_f$.
\end{enumerate}
In this case, $A = H(Y,\EE)$.
\end{proposition}
\begin{proof}
\underline{$(1)\Rightarrow(2)$}.~ We can write $A = H(Y',\EE)$ for $Y'$ a sober topological space and $\EE$ a sheaf on $Y'_f$. Then $\OO(Y') = A/\mathcal{D} = \OO(Y)$, so $Y$ is the underlying locale of the sober topological space $Y'$. We make the identification $Y = Y'$. Take $k$ as in (\ref{eq:pullback-dissolution-3}). By Theorem \ref{thm:duality}, the pair $(Y,k_*\EE)$ defines a noncommutative frame $A'$, and using equations (\ref{eq:duality-1}) and (\ref{eq:duality-2}) we see that $A' \cong A$. This shows that $\FF \cong k_*\EE$.

\underline{$(2)\Rightarrow(1)$}.~ Suppose that $Y$ is spatial and that $\FF \cong k_*\EE$ for some sheaf $\EE$ on $Y_f$. Then we can use equations (\ref{eq:duality-1}) and (\ref{eq:duality-2}) to show that $A = H(Y,\FF)$.
\end{proof}

\bibliographystyle{amsalphaarxiv}
\bibliography{thesis/thesis}

\newcommand{\etalchar}[1]{$^{#1}$}
\providecommand{\bysame}{\leavevmode\hbox to3em{\hrulefill}\thinspace}
\providecommand{\MR}{\relax\ifhmode\unskip\space\fi MR }
\providecommand{\MRhref}[2]{%
  \href{http://www.ams.org/mathscinet-getitem?mr=#1}{#2}
}
\providecommand{\href}[2]{#2}
\begin{thebibliography}{BCVG{\etalchar{+}}13}

\bibitem[ABMZ19]{frame-of-nuclei-spatial}
Francisco \'Avila, Guram Bezhanishvili, Patrick Morandi, and Angel Zald\'ivar,
  \emph{When is the frame of nuclei spatial: A new approach}, \href
  {http://arxiv.org/abs/1906.03636} {\path{arXiv:1906.03636}}.

\bibitem[Ban96]{banaschewski}
B.~Banaschewski, \emph{Radical ideals and coherent frames}, Comment. Math.
  Univ. Carolin. \textbf{37} (1996), no.~2, 349--370. \MR{1399006}

\bibitem[BCVG{\etalchar{+}}13]{kcv-priestley}
Andrej Bauer, Karin Cvetko-Vah, Mai Gehrke, Samuel~J. van Gool, and Ganna
  Kudryavtseva, \emph{A non-commutative {P}riestley duality}, Topology Appl.
  \textbf{160} (2013), no.~12, 1423--1438. \MR{3072705}

\bibitem[BL95]{jl-discriminator}
Robert Bignall and Jonathan Leech, \emph{Skew {B}oolean algebras and
  discriminator varieties}, Algebra Universalis \textbf{33} (1995), no.~3,
  387--398. \MR{1322781}

\bibitem[CC14]{connes-consani}
Alain Connes and Caterina Consani, \emph{The arithmetic site}, C. R. Math.
  Acad. Sci. Paris \textbf{352} (2014), no.~12, 971--975. \MR{3276804}

\bibitem[Cor75]{cornish-priestley}
William~H. Cornish, \emph{On {H}. {P}riestley's dual of the category of bounded
  distributive lattices}, Mat. Vesnik \textbf{12(27)} (1975), no.~4, 329--332.
  \MR{0398934}

\bibitem[CV19]{kcv-frames}
Karin Cvetko-Vah, \emph{Noncommutative frames}, J. Algebra Appl. \textbf{18}
  (2019), no.~1, 1950011, 13. \MR{3910664}

\bibitem[CVHL19]{completeness-issues}
Karin Cvetko-Vah, Jens Hemelaer, and Jonathan Leech, \emph{{Noncommutative
  Frames Revisited}}, preprint (2019), \href {http://arxiv.org/abs/1911.12355}
  {\path{arXiv:1911.12355}}.

\bibitem[CVHLB19]{nctop}
Karin Cvetko-Vah, Jens Hemelaer, and Lieven Le~Bruyn, \emph{What is a
  noncommutative topos?}, J. Algebra Appl. \textbf{18} (2019), no.~6, 1950107,
  18. \MR{3954661}

\bibitem[Hoc69]{hochster}
Melvin Hochster, \emph{Prime ideal structure in commutative rings}, Trans.
  Amer. Math. Soc. \textbf{142} (1969), 43--60. \MR{0251026}

\bibitem[Joh82]{johnstone-stone}
Peter~T. Johnstone, \emph{Stone spaces}, Cambridge Studies in Advanced
  Mathematics, vol.~3, Cambridge University Press, Cambridge, 1982. \MR{698074}

\bibitem[Joh02a]{johnstone-elephant-1}
\bysame, \emph{Sketches of an elephant: a topos theory compendium. {V}ol. 1},
  Oxford Logic Guides, vol.~43, The Clarendon Press, Oxford University Press,
  New York, 2002. \MR{1953060}

\bibitem[Joh02b]{johnstone-elephant-2}
\bysame, \emph{Sketches of an elephant: a topos theory compendium. {V}ol. 2},
  Oxford Logic Guides, vol.~44, The Clarendon Press, Oxford University Press,
  Oxford, 2002. \MR{2063092}

\bibitem[Kli13]{klinke-slides}
Olaf Klinke, \emph{{The assembly, Smyth's stable compactifications and the
  patch frame (Talk at BLAST 2013, slides and video)}}, URL:
  \url{http://math.chapman.edu/~jipsen/blast2013/speakers.html}.

\bibitem[LB16]{llb-covers-general-version}
Lieven Le~Bruyn, \emph{Covers of the arithmetic site}, preprint (2016), \href
  {http://arxiv.org/abs/1602.01627} {\path{arXiv:1602.01627}}.

\bibitem[Lee89]{jl-rings}
Jonathan Leech, \emph{Skew lattices in rings}, Algebra Universalis \textbf{26}
  (1989), no.~1, 48--72. \MR{981425}

\bibitem[Lee90]{jl-boolean}
\bysame, \emph{Skew {B}oolean algebras}, Algebra Universalis \textbf{27}
  (1990), no.~4, 497--506.

\bibitem[Lee92]{jl-normal}
\bysame, \emph{Normal skew lattices}, Semigroup Forum \textbf{44} (1992),
  no.~1, 1--8. \MR{1138681}

\bibitem[Lee96]{jl-survey}
\bysame, \emph{Recent developments in the theory of skew lattices}, Semigroup
  Forum \textbf{52} (1996), no.~1, 7--24. \MR{1363525}

\bibitem[Lee19]{jl-journey}
\bysame, \emph{My journey into noncommutative lattices and their theory}, The
  Art of Discrete and Applied Mathematics \textbf{2} (2019), no.~2, 1--19.

\bibitem[MLM94]{maclane-moerdijk-sheaves}
Saunders Mac~Lane and Ieke Moerdijk, \emph{Sheaves in geometry and logic},
  Universitext, Springer-Verlag, New York, 1994, A first introduction to topos
  theory, Corrected reprint of the 1992 edition. \MR{1300636}

\bibitem[PP12]{picado-pultr}
Jorge Picado and Ale\v{s} Pultr, \emph{Frames and locales}, Frontiers in
  Mathematics, Birkh\"{a}user/Springer Basel AG, Basel, 2012, Topology without
  points. \MR{2868166}

\bibitem[Pri70]{priestley-duality}
Hilary~A. Priestley, \emph{Representation of distributive lattices by means of
  ordered stone spaces}, Bull. London Math. Soc. \textbf{2} (1970), 186--190.
  \MR{0265242}

\bibitem[Pri94]{priestley}
\bysame, \emph{Spectral sets}, J. Pure Appl. Algebra \textbf{94} (1994), no.~1,
  101--114. \MR{1277526}

\bibitem[PS88]{pultr-sichler}
Ale\v{s} Pultr and Jir\'i Sichler, \emph{Frames in {P}riestley's duality},
  Cahiers Topologie G\'{e}om. Diff\'{e}rentielle Cat\'{e}g. \textbf{29} (1988),
  no.~3, 193--202. \MR{975372}

\bibitem[Sim80]{simmons}
Harold Simmons, \emph{Spaces with {B}oolean assemblies}, Colloq. Math.
  \textbf{43} (1980), no.~1, 23--39 (1981). \MR{615967}

\end{thebibliography}

\end{document}